\documentclass{article}
\usepackage[margin = 4.2cm]{geometry}
\usepackage{amsmath,amsthm,amssymb}
\providecommand{\R}{\mathbb R}
\providecommand{\N}{\mathbb N}
\providecommand{\E}{\mathop{\mathsf{E}{}}\nolimits}
\providecommand{\prob}{\mathop{\mathsf{P}{}}\nolimits}
\providecommand{\cov}{\mathop{\mathrm{cov}{}}\nolimits}
\providecommand{\e}{\mathrm{e}}
\providecommand{\abs}[1]{\lvert #1 \rvert}

\newtheorem{thm}{Theorem}
\newtheorem{lem}[thm]{Lemma}
\newtheorem{prop}[thm]{Proposition}
\newtheorem{cor}[thm]{Corollary}
\theoremstyle{definition}
\newtheorem{defi}[thm]{Definition}
\theoremstyle{remark}
\newtheorem*{rem}{Remark}
\newtheorem*{rems}{Remarks}
\title{Cover times and generic chaining}
\author{Joseph Lehec\footnote{Universit\'e Paris-Dauphine}}

\begin{document}

\maketitle

\begin{abstract}
A recent result of Ding, Lee and Peres expresses 
the cover time of the random walk on a graph in terms 
of generic chaining for the commute distance. 
Their proof is very involved and the purpose
of this article is to present a simpler approach to this problem
based on elementary hitting times estimates and
chaining arguments. Unfortunately we fail
 to recover their full result, but not by much.
 \end{abstract}

%
%Using the Markov property, it is easy to see that $d$ is a distance. 
%The Markov 
%%Moreover this distance is Euclidean: if $G$ is the Green function 
%%of the Markov process killed at some point $x_0$ then
%%\[
%%
%%\] 
%Since the square root is subadditive $\sqrt{d}$ is again a distance.
%Ding, Lee and Peres prove the following. 
%%
%\begin{thm}
%For any finite connected graph 
%we have
%\[
%\max_{x\in G} \Bigl(  \E_x  T_{cov} \Bigr) \sim \Bigl( \gamma_2 ( G , \sqrt{d} ) \Bigr)^2 . 
%\]
%\end{thm}
%
\section{Introduction}
 Let $(X_n)_{n\geq 0}$ be an irreducible Markov chain on some
 state space $M$. Given $A\subset M$ let
\[
T (A) = \inf \{ n \geq 0 \colon X_n  \in A\}
\]
be the first time the chain hits $A$ and let
\[
T_{cov} (A) = \sup_{x\in A} T (x) 
\]
be the first time the chain $X$ has visited every point of $A$. 
The cover time of $A$ is by definition
\[
\cov ( A ) = \sup_{x\in A} \bigl( \E_x T_{cov} (A) \bigr)  .
\]
%As usual $\E_x$ e subscript $x$ means  we have
%$X_0 =x$ almost surely. 
To avoid trivial situations, the chain is 
assumed to be positive recurrent throughout
so that $\cov ( A) <+\infty$ if and only if $A$ is finite. 
\\
Using the strong Markov property it is easily 
seen that given $x,y,z$ in $M$
\[
\E_x T(y) + \E_y T(z)
\]
is the expectation (under $\prob_x$) of the first time
the chain has visited $y$ and $z$ (in this order). 
This implies that
\[
\E_x T(y) + \E_y T(z) \geq \E_x T(z). 
\]
Therefore the commute time
\[
d(x,y)= \E_x T(y) + \E_y T(z) 
\]
is a distance on $M$. 
This article deals with 
the following problem,
dating back to Matthews' article~\cite{MAT}
at least: can $\cov(A)$ be estimated in terms
of metric properties of $(A,d)$?
An arguably definitive answer 
to this question has recently 
be given by Ding Lee and Peres~\cite{DLP}, 
their result is expressed in terms
of generic chaining. 
%%%%%%%%%%%%%%%%%%%%%%%%
\subsection*{The generic chaining}
The generic chaining is a tool designed
by Talagrand to estimate suprema of Gaussian processes. 
Let us describe it briefly and refer to the book~\cite{TALA} for details.\\
Throughout we let $(N_n)_{n\geq 0}$ be the following sequence of integers:
\begin{equation}
\label{definitionNn}
N_0 =1 , \ N_n = 2^{2^n} , n\geq 1 .
\end{equation}
Given a set $M$, a sequence $(\mathcal A_n)_{n\geq 0}$ of partitions
of $M$ is called admissible if
 $\mathcal A_{n+1}$ is a refinement of $\mathcal A_n$
 and if  $\abs{ \mathcal A_n } \leq N_n$ for every $n \geq 0$,
 where $\abs{ \mathcal A_n}$ is just the cardinality of $\mathcal A_n$.  
The cardinality condition implies in particular that $\mathcal A_0 = \{M\}$.  
Given a sequence of partition $(\mathcal A_n)_{n\geq 0}$ of
$M$ and $x\in M$ we let $A_n (x)$ be the only element of $\mathcal A_n$
containing $x$.  
\begin{defi}
Let $(M,d)$ be a metric space. Set
\[
\gamma_2 ( M , d ) 
= \inf \Bigl[ \sup_{x\in M} \Bigl( \sum_{n=0}^{+\infty} 2^{n/2} \Delta ( A_n (x) , d ) \Bigr) \Bigr] , 
\]
where the infimum is taken over all admissible partitions $(\mathcal A_n)_{n\geq 0}$
of $M$, and $\Delta ( A,d)$ denotes the diameter of $A$. 
\end{defi}
Recall that a Gaussian process is a family $(T(x))_{x\in M}$ of
random variables such that every linear combination of the variables $T(x)$
is Gaussian. The process is said to be centered if $\E T(x) = 0$ for every $x$. 
The fundamental result of Talagrand reads as follows.
\begin{thm}
Let $(T(x))_{x\in M}$ be a centered Gaussian process. 
Then 
\begin{equation}
\label{talagrand}
\frac 1 L \gamma_2 ( M ,d ) \leq \E \sup_{x\in M} T(x) \leq 
L \gamma_2 ( M,d )
\end{equation}
where $L$ is a universal constant and $d$ is the following distance on $M$
\begin{equation}
\label{L2distance}
d ( x, y ) = \sqrt{ \E (T(x) - T(y))^2} .
\end{equation}
\end{thm}
The upper bound is not
specific to Gaussian processes,
it applies to any
centered process $(T(x))_{x\in M}$ satisfying
\begin{equation}
\label{gaussian-deviation}
\prob ( T(x) - T(y) \geq u ) \leq \e^{ - u^2 / 2 d(x,y)^2}
\end{equation}
for all $x,y\in M$, for all $u>0$ and for some distance $d$. 
Using a union bound it is not hard to see that a centered process 
satisfying~\eqref{gaussian-deviation} satisfy
\begin{equation}
\label{majoration-gaussienne}
\E \sup_{x\in A} T(x) \leq C \sqrt{ \log \abs{ A } } \max_{x,y \in A} d(x,y) , 
\end{equation}
for every finite subset $A$ of $M$. The proof
of the upper bound of~\eqref{talagrand}
consists in applying this union 
bound repeatedly and at different scales. 
\\
The lower bound is another story, it is 
specific to Gaussian processes 
and much more difficult to prove. 
The key estimate is the Sudakov inequality:
if $(T(x))_{x\in M}$ is a centered Gaussian
process then for all finite 
subset $A$ of $M$ 
\begin{equation}
\label{sudakov}
\E \sup_{x\in A} T(x) \geq  c \sqrt{\log\abs{A}} \min_{x\neq y \in A} d(x,y) 
\end{equation}
where $c$ is a universal constant and $d$ is the $L^2$-distance~\eqref{L2distance}. 
%%%%%%%%%%%%%%%%%%%%%%%%%%%%%%%%%%%%%%%%%%%%%
\subsection*{The Ding, Lee and Peres theorem}
Cover times satisfy inequalities analogue to~\eqref{majoration-gaussienne} and~\eqref{sudakov}
due to Matthews~\cite{MAT}: for any finite subset $A$ of $M$
\begin{equation}
\label{matthews-eq}
\begin{split}
\cov ( A ) & \leq ( 1 + \log  \abs{A}  ) \max_{x,y \in A} \bigl( \E_x T(y) \bigr)\\
\cov ( A ) & \geq \log \abs{A}  \min_{x\neq y \in A} \bigl( \E_x T(y) \bigr). 
\end{split}
\end{equation}
In view of these inequalities it seems natural to conjecture that the correct order of magnitude
for $\cov ( A )$ is 
\[
\gamma_1 ( A , d ) 
= \inf \Bigl[ \sup_{x\in A} \Bigl( \sum_{n=0}^{+\infty} 2^{n} \Delta ( A_n (x) , d ) \Bigr) \Bigr] , 
\]
rather than $\gamma_2 ( A, d )$ (recall that $d$ is the commute
distance $d(x,y)= \E_x T(y) + \E_y T(z)$). 
This is not quite correct. Here is the result of Ding, Lee and Peres.
\begin{thm}
\label{DLPthm}
If the Markov chain $X$ is reversible (and if the state space $M$ is finite) then
\[
\frac 1 L \big[ \gamma_2 ( M , \sqrt{d} ) \bigr]^2
\leq \cov ( M ) 
\leq L   \big[ \gamma_2 ( M , \sqrt{d} ) \bigr]^2
\]
for some universal constant $L$. 
\end{thm}
\begin{rem}
Actually the inequality remains valid 
when $M$ is infinite. Indeed since $d(x,y)\geq 1$ 
when $x\neq y$, we then have $\gamma_2 ( M , \sqrt{d} ) = +\infty$. 
\end{rem}
The correct order of magnitude $\gamma_2 ( M , \sqrt{d} )^2$ is comparable
to our wrong guess: clearly
\[
\gamma_1 ( M , d) \leq \bigl[ \gamma_2 ( M ,\sqrt{d} ) \bigr]^2 . 
\]
%%%%%%%%%%%%%%%%%%%%%%%%%%%%%%%%%%%%%%%%%%%%%%%
\subsection*{Purpose of the present article}
%
%
%revoir ce paragraphe
%
The proof of Ding, Lee and Peres is very involved. 
In particular it relies on the Ray-Knight isomorphism
theorem which makes a connection 
between local times of the chain 
and the Gaussian free field
associated to the chain. 
It may be interesting  to 
have a simpler proof 
relying only on elementary
hitting times estimates and on Talagrand's generic chaining. 
The purpose of this article is to provide such a proof. 
\\
Unfortunately we fail to recover entirely Theorem~\ref{DLPthm}, here is what we prove. 
\begin{thm}
\label{main-upper}
If $X$ is irreducible and positive recurrent, then
\begin{equation}
\label{non-heredite}
\cov ( M ) \leq L \bigl[ \gamma_2 ( M , \sqrt{d} ) \bigr]^2 . 
\end{equation}
for some universal $L$. More generally we have
\begin{equation}
\label{heredite}
\cov ( A ) \leq L \bigl[ \gamma_2 ( A , \sqrt{d} ) \bigr]^2 , 
\end{equation}
for every subset $A$ of $M$. 
\end{thm}
Inequality~\eqref{non-heredite} is slightly
stronger than the upper bound of Theorem~\ref{DLPthm} 
since the chain is no longer assumed to be reversible. 
Besides, it is not clear whether the approach of Ding, Lee
and Peres yields~\eqref{heredite}. 
\begin{thm}
\label{main-lower}
If in addition the chain $X$ is reversible then
\begin{equation}
\label{main-lower-eq}
\gamma_1 ( M , d ) \leq L \cov (M) ,
\end{equation}
where $L$ is a universal constant. Again 
we actually have
\[
\gamma_1 ( A , d ) \leq L \cov ( A ) ,
\] 
for every $A \subset M$.
\end{thm}
\begin{rem}
The reversibility assumption is necessary. Indeed,
consider the discrete torus $\mathbb Z _N$ and the Markov
kernel given by
\[
P ( x , x+1 ) = 1 , \quad \forall x\in\mathbb Z_N . 
\]
Clearly $d(x,y) = N$ for all $x\neq y$, which implies that
\[
\gamma_1 ( T,d ) \approx N \log( N) .
\]
On the other hand $T_{cov} ( \mathbb Z_N ) = N$ p.s. (whatever the starting point).
\end{rem}
Since $\gamma_1 ( M,d) \leq  [ \gamma_2 ( M, \sqrt{d} ) ]^2$ 
inequality~\eqref{main-lower-eq} is weaker than the lower bound of Theorem~\ref{DLPthm}.
Let us comment a little bit more on this. 
In order to compute $\gamma_1 ( M ,d ) $ one can restrict to 
partitions $(\mathcal A_n)_{n\geq 0}$ satisfying
\[
\mathcal A_n  = \{ \{x\} , \ x\in M\} 
\]
for $n\geq k$, where $k$ is the only integer satisfying
\[
N_{k-1}  < \abs{M} \leq N_k . 
\]
Then by convexity we get
\[
\begin{split}
\Bigl( \sum_{n=0}^\infty 2^{n/2} \sqrt{ \Delta ( A_n(x) , d) }  \Bigr)^2
 & = \Bigl( \sum_{n=0}^k 2^{n/2} \sqrt{ \Delta ( A_n(x) , d)  } \Bigr)^2 \\
 & \leq (k+1) \sum_{n=0}^\infty 2^{n} \Delta ( A_n(x) , d) 
 \end{split}
\]
for every $x\in M$, yielding
\[
\bigl[ \gamma_2 ( M , \sqrt{d} ) \bigr]^2 
\leq C \log ( \log \abs{M}  ) \gamma_1 ( M ,d)
\]
for some universal $C$ (provided $\abs{M} \geq 3$).
Therefore the estimate~\eqref{main-lower-eq} 
is off the correct order
of magnitude by at most a factor $\log ( \log  \abs{M}  )$.  
This is sharp, there is 
a Markov chain for which the gap is indeed $\log ( \log  \abs{M}  )$
(see the appendix).
%Moreover this is sharp: it is possible to find a reversible 
%Markov chain for which 
%$\gamma_2 ( M,  \sqrt{d})^2$
%is of the order
%
%and 
%It is not clear whether~\eqref{comparison-gamma}
%is sharp. In every situation that the author is aware of, 
%there is a dominating term is the sum
%\[
%\sum_{n=0}^{+\infty} 2^n \Delta ( A_n (x) , d) 
%\] 
%so that the convexity inequality used previously is not sharp. 
%Actually we were not even able to rule out the 
%the possibility for~\eqref{main-lower-eq} to be sharp. 
%%%%%%%%%%%%%%%%%%%%%%%%%%%%%%%%%%%%%%%%%%%%%%%%%%%%
\section{The upper bound}
Since $(X_n)_{n\geq 0}$ is 
an irreducible, positive recurrent Markov chain,
there is a unique invariant 
probability measure which we denote by $\pi$. 
The purpose of this section is to bound 
\[
\E \sup_{x\in M} T (x) 
\]
through a chaining argument. Since
no estimate such as~\eqref{gaussian-deviation}
is available for hitting times, the chaining 
procedure will be different from Talagrand's,
and is taken from the articles~\cite{BDNP,KKLV}. 
\\
We need a couple of additional notations. 
Recall that 
\[
T(x) = \inf( n \geq 0 , \ X_n = x )
\]
is the hitting time of $x$. 
Let 
\[
T^1 (x) = \inf( n \geq 1 , \ X_n = x ) 
\]
be the first return time to $x$ and for $k\geq 2$ define inductively
the $k$-th return time to $x$ by
\[
T^k (x) = \inf ( n \geq T^{k-1} (x)  +1 , \ X_n = x ) .
\]
We also let $T^0 (x) =0$ by convention. 
Lastly, let 
\[
N_k = \sum_{n=0}^{k-1} \delta_{X_n} 
\]
be the empirical measure of the chain $X$. In other
words $N_k(x)$ is the number of visits to $x$ before time $k$.
\\
The following deviation estimate is 
due to Kahn, Kim, Lovasz and Vu~\cite{KKLV}.
\begin{lem}
\label{KKLV}
Let $x\neq y$ in $M$. Then for every $\epsilon >0$
and for every integer $k$
\[
\prob_x \Bigl( N_{T^k (x)} (y) \leq (1-\epsilon) \frac{k \pi(y)}{\pi(x)} \Bigr)
\leq 
\exp\Bigl[ - \frac{\epsilon^2 k}{4 \pi(x) d(x,y)} \Bigr] .
\]
\end{lem}
Let us sketch the argument. Because
of the strong Markov property, under $\prob_x$
the variables
\[
\bigl( N_{T^i (x)} (y) - N_{T^{i-1} (x)} (y) \bigr)_{i\geq 1}
\]
are independent and identically distributed. 
And it is a standard fact (see for instance~\cite[chapter 2]{AF}) that their law is geometric:
for every integer $r$ 
\[
\prob_x \bigl( N_{T^1(x)} (y) \geq r \bigr) = p_{xy} (1-p_{yx})^r ,
\]
where 
\[
p_{xy} = \prob_x (  T(y) \leq T^1 (x) ) = \frac1 {\pi (x) d(x,y)} . 
\]
The previous lemma is thus a Hoeffding type estimate for sums of independent geometric
variables. We refer to~\cite{KKLV} for the details. 
\\
Our next tool is taken from Barlow, Ding, Nachmias and Peres~\cite{BDNP}.
\begin{lem}
\label{BDNPlemma}
Let $A$ be a finite subset of $M$, let $z \in A$ and let $k$ be an integer. Then
\[
\begin{split}
\E_z T_{cov} (A) 
 & \leq \frac{ \E_z T^k (z)  }{Ê\prob_z \bigl( T_{cov} (A)  \leq T^k (z) \bigr) }  \\
&= \frac k   { \pi (z) Ê\prob_z \bigl( T_{cov} (A)  \leq  T^k (z) \bigr) }   
\end{split}
\]
\end{lem}
Again we sketch the proof and refer to~\cite{BDNP}
for details. 
Let
\[ N = \inf ( n \geq 1 , \ T_{cov} (A) \leq T^{nk} (z) ) . \]
Then by Wald's identity
\[
\E_z T_{cov} (A) \leq  \E_z T^{Nk} (z)  = \E_z ( N ) \E_z T^k (z) .
\] 
On the other hand if $N$ is larger than $n$ then the walk fails 
to cover $A$ during any of the following intervals of time
\[ [0,T^k(z)),\ [T^k(z)  , T^{2k} (z)),\dotsc,\ [T^{(n-1)k } (z)  , T^{nk}(z)  ) \]
so that
\[
\prob_x ( N > n ) \leq \prob_z \bigl( T_{cov} (A)  \geq T^k (z)  \bigr)^n .
\]
The result follows. 
\\
The authors of~\cite{BDNP} combine these two lemmas 
with a nice chaining argument. 
Although it is not written 
this way, their result is essentially the Dudley 
version of Theorem~\ref{main-upper}:
\begin{equation}
\label{dudley}
\cov (M ) \leq L \bigl( \sum_{n=0}^\infty \e_n ( M , \sqrt{d}) 2^{n/2} \bigr)^2
\end{equation}
where 
\[
\e_n (M , \sqrt{d} ) = \inf_{A} \bigl( \sup_{x\in M} \sqrt{ d(x, A) } \bigr)
\]
where the infimum is taken over all subsets $A$ of $M$ 
satisfying $\abs{A} \leq N_n$. This is weaker than
Theorem~\ref{main-upper}. 
Indeed swapping the sup and the sum
in the definition of $\gamma_2$, it is easily seen that
\[
\gamma_2 ( T , \sqrt{d} )  \leq C \sum_{n=0}^\infty \e_n ( M , \sqrt{d} ) 2^{n/2} ,
\]
for some universal constant $C$. 
We show that it is possible to modify BDLP's chaining argument to
obtain Theorem~\ref{main-upper}.
\\
Let $z,x,y$ in $M$ such that $x\neq y$ and let $k,l$ be two integers larger than $1$. 
Observe that 
\[
\begin{split}
\prob_z ( T^l (y) > T^k(x) ) 
& = \prob_z ( N_{T^k (x) } (y) \leq l-1 ) \\
& \leq \prob_z   ( N_{T^k (x) } (y)  - N_{ T^1 (x)} (y) \leq l-1 ) \\
& = \prob_x  ( N_{T^{k-1} (x) } (y)  \leq l-1 ) . 
\end{split}
\]
The latest equality being a consequence of the strong Markov property. 
If $(l-1) / \pi (y) < (k-1) / \pi (x)$, applying Lemma~\ref{KKLV} to $k-1,l-1$ and
\[
\epsilon = 1 - \frac{ (l-1) \pi(x) }{ (k-1) \pi(y) } 
\]
gives
\begin{equation}
\label{majoration-hitting}
\prob_z  ( T^l (y) > T^k (x) ) \leq
 \exp\left[ - \frac{ \bigl( \frac{k-1}{\pi(x)} - \frac{l-1}{\pi(y)}  \bigr)^2 }{ 4 d(x,y) \frac{k-1}{\pi(x)} } \right] .
\end{equation}
This will be our key estimate. Lastly, we shall use the following elementary fact:
if $x$ and $y$ are distinct elements of $M$ then
\[
\frac 1 {\pi (x)}  = \E_x T^1 (x) \leq \E_x T(y) + \E_y  T(x) = d(x,y) . 
\]
Let us reformulate Theorem~\ref{main-upper}. 
\begin{prop}
Let $A \subset M$, let $z\in A$ and let $(\mathcal A_n)_{n\geq 0}$ 
be an admissible sequence of partitions of $A$. Then 
\[
\E_z ( T_{cov} (A) ) \leq L \Bigl( \sup_{x\in A} \sum_{n=0}^\infty 2^{n/2} \sqrt{\Delta (A_n(x))} \Bigr)^2 . 
\] 
Recall that $A_n (x)$ denotes the only element of $\mathcal A_n$ containing $x$. Also
$\Delta$ denotes the diameter with respect to the commute distance. 
\end{prop}
\begin{proof}
Let $t_0 ( A  ) = z$, and for each $n$ and for each $ B \in \mathcal A_n $ let $t_n (B)$ be an arbitrary
element of $B$. Given $x \in A$, we let $x_n = t_n ( A_n(x) )$. We can assume that $A$ is finite and 
that 
\[
\mathcal A_n = \{ \{ x \} , \ x \in A\}
\]
for $n$ large enough (the right hand of the desired inequality equals $+\infty$ otherwise). 
Therefore $x_n = x$ eventually. 
Let 
\[
r_n (x) =  \sup_{y \in A_n (x)}  \sum_{k=n}^{+\infty} 2^{k/2}  \sqrt{ \Delta ( A_k (y) ) } 
\]
and 
\[
k_n (x) = \lfloor 34 \cdot \pi (x_n) r_n(x) r_0 (x) \rfloor + 1 ,
\]
where $\lfloor r \rfloor$ denotes the integer part of $r$. 
Observe that $r_n(x)$ and $k_n (x)$ depend only on $A_n(x)$. In particular
$k_0 (x)$ depends on nothing. Also
\[
\begin{split}
r_n (x) - r_{n+1} (x) 
& \geq 2^{n/2} \sqrt{ \Delta ( A_n (x) ) } \\
& \geq 2^{n/2} \sqrt{ d(x_n , x_{n+1} ) }.
\end{split}
\]
We claim that for every $x$ and $n$
\begin{equation}
\label{step-upper}
\prob_z \bigl(  T^{k_{n+1} ( x )} ( x_{n+1} )  > T^{ k_n ( x ) } (x_n) \bigr)
\leq \e^{-2^{n+3}} \leq \frac 1 { N_{n+3} } .
\end{equation}
Indeed, if $x_n = x_{n+1}$ then $k_{n+1} ( x ) \leq k_n ( x )$
and the inequality is trivial. Otherwise write
\[
\begin{split}
\frac{ k_n ( x)  - 1} { \pi (x_n) } - \frac{ k_{n+1} (x) - 1} { \pi (x_{n+1}) } 
& \geq 34 \cdot ( r_n(x) - r_{n+1} (x) ) r_0 (x) - \frac 1 {\pi (x_n)}  \\
& \geq 34 \cdot 2^{n/2} \sqrt{d (x_n , x_{n+1} )} r_0 (x) - \frac 1 {\pi (x_n)} . 
\end{split}
\]
Since $x_n \neq x_{n+1}$ and $\sqrt{d(x_n , x_{n+1})}\leq r_0 (x)$ we have
\[
\frac 1 {\pi (x_n)} \leq \sqrt{d (x_n , x_{n+1} )}  r_0 (x)  . 
\]
Therefore 
\[
\begin{split}
\frac{ k_n ( x) - 1  } { \pi (x_n) } - \frac{ k_{n+1} (x) - 1} { \pi (x_{n+1} )} 
& \geq (  34 \cdot 2^{n/2} - 1 ) \sqrt{d (x_n , x_{n+1} )} r_0 (x) \\
& \geq 33 \cdot 2^{n/2}  \sqrt{d (x_n , x_{n+1} )} r_0 (x) . 
\end{split}
\]
Also
\[
\frac{ k_n ( x) - 1  } { \pi (x_n) } \leq 34 \cdot r_n (x) r_0 (x) \leq 34 \cdot  r_0 (x)^2 . 
\] 
Since $33^2 / (4 \cdot 34 ) \geq 2^3$, combining~\eqref{majoration-hitting}
with the last two inequalities yields~\eqref{step-upper}.
\\
The number  of possible couples $(x_n,x_{n-1})$ is at most $N_n N_{n+1}$.
Recall the definition~\eqref{definitionNn} of $N_n$ and observe that $N_n^2 \leq N_{n+1}$
for all $n$. A union bound shows that the probability that there exists $x$ and $n$ such that
\[
 T^{k_{n+1} ( x )} (x_{n+1})  \geq T^{k_n (x)} (x_n)
\]
is at most
\[
\sum_{n\geq 0} \frac{ N_n N_{n+1} }{ N_{n+3} } \leq \sum_{n\geq 0} \frac 1{N_{n+2}} \leq \sum_{n\geq 4} 2^{-n} 
= \frac 1 8. 
\]
Therefore with probability at least $7/8$, we have
\[ 
T^{k_{n+1} ( x )} (x_{n+1})  \leq T^{k_n ( x )} (x_n)
\]
for all $x$ and $n$, hence
\[
T^{k_n ( x)} (x_n) \leq T^{ k_0 ( x )} (x_0) = T^{k_0 } (z) .
\]
Since $x_n = x$ for $n$ large enough and $k_n (x) \geq 1$ we obtain
\[
\forall x\in A, \ T (x) \leq T^{k_0} (z) 
\]
with probability $7/8$ at least. In other words
\[
\prob_z ( T_{cov} (A)  \leq T^{ k_0} (z) ) \geq \frac 7 8 .
\]
Together with Lemma~\ref{BDNPlemma} we get
\[
\E_z  T_{cov} (A) 
\leq  \frac { 8 k_0 }{ 7 \pi ( z ) }
\leq \frac 8 7 ( 34 \cdot r_0^2 + \frac 1 {\pi (z)} ) .
\]
Unless $A = \{ z \}$, in which case $\cov ( A ) = 0$
and there is nothing to prove, we have $1/ \pi ( z ) \leq \Delta ( A ) \leq r_0^2$. 
Therefore 
\[
\E_z T_{cov} ( A ) \leq \frac{ 8 \cdot 35} 7  \bigl( \sup_{x\in A} \sum_{n=0}^\infty 2^{n/2} \sqrt{\Delta ( A_n (x) )} \Bigr)^2 . 
\qedhere
\]
\end{proof}
%
%%%%%%%%%%%%%%%%%%%%%%%%%%%%%%%%%%%%%%%%%%%%%
%
\section{The lower bound}
First let us slightly modify the definition 
of cover time: given $A \subset M$ let
\[
\begin{split}
\cov_- ( A ) & = \min_{x\in A} \E_x  T_{cov} (A) \\
\cov_+ ( A ) & = \max_{x\in A} \E_x  T_{cov} (A) 
\end{split}
\]
In this section we prove the following
\begin{prop}
Let $X$ be an irreducible, positive recurrent 
Markov chain on a discrete state space $M$. 
If the chain is reversible then for every finite
subset $A$ of $M$
\[
\gamma_1 ( A , d ) \leq L ( \cov_- ( A ) + \Delta ( A , d ) ) ,
\]
where $L$ is a universal constant. 
\end{prop}
\begin{rems}
\begin{enumerate}
\item
This yields Theorem~\ref{main-lower} since clearly
\[
\begin{split}
\cov_- ( A ) & \leq \cov_+ ( A ) \\
\Delta ( A , d ) & \leq \cov_+ ( A ).
\end{split}
\]  
\item
The term $\Delta ( A,d )$ cannot be removed from the inequality.
Indeed if $M = \{0,1\}$ and 
the transitions are given by the matrix
\[
\begin{pmatrix}
\epsilon & 1-\epsilon\\
\epsilon & 1 - \epsilon\\
\end{pmatrix}
\]
then 
\[
\gamma_1 ( M, d) \geq \Delta ( M ,d) =\frac 1 { \epsilon ( 1- \epsilon ) }
\]
whereas $\cov_- (M ) = \min ( \frac1 \epsilon  , \frac 1 {1-\epsilon} )$. 
\end{enumerate}
\end{rems}
\subsection*{Talagrand's growth condition}
Recall the majoring measure theorem: if
$(T(x))_{x \in M}$ is a centered Gaussian process
then 
\[ \gamma_2 ( T, d ) \leq L \E \sup_{x\in M} T(x) , \]
where $d$ is the $L^2$ distance~\eqref{L2distance}.
The proof of Talagrand consists in showing
(using Sudakov's inequality) that the functional
\[
A \mapsto  \E \sup_{x\in A} T(x)
\]
satisfies an abstract growth condition, 
and that such functionals dominate $\gamma_2$.
Here is the definition of the growth condition 
adapted to the $\gamma_1$ situation (rather than $\gamma_2$). 
\begin{defi}[Growth condition]
\label{growth-def}
Let $(M,d)$ be a metric space. 
A functional $F \colon \mathcal P (M) \to \R_+$ is said to satisfy the growth condition with parameters $r >1$ and $\tau \in \mathbb N$ if for every step $n \in \mathbb N$ and every scale $a>0$ the followings holds. Let $m = N_{n+\tau}$,
for every sequence $H_1 , \dotsc , H_m$ of non-empty subsets of $M$ satisfying
\begin{enumerate}
\item $\Delta ( \cup_{i\leq m}  H_i ) \leq r a$,
\item $ d ( H_i , H_j ) \geq a$ for all $i\neq j$,
\item $\Delta ( H_i ) \leq a/r$ for all $i$,
\end{enumerate}
we have 
\[
F ( \cup_{i\leq m} H_i ) \geq a 2^{n} + \min_{i\leq m} F( H_i ) .
\]
\end{defi}
\begin{thm}
\label{growth-thm}
If $F$ is a non-decreasing for the inclusion and satisfies the growth condition with parameters $r$ and $\tau$ then 
\[
\gamma_1 ( M , d ) \leq L 2^{\tau} ( \Delta ( M , d) + r F( M ) ) ,
\]
where $L$ is a universal constant. 
\end{thm}
We refer to~\cite{TALA} for a proof of this theorem. 
The purpose of the rest of this section is
to show that the functional 
\[
A \mapsto \cov_- (A) 
\]
is non-decreasing and satisfies the growth condition
on $(M,d)$ (where $d$ is the commute distance)
with universal parameters $\tau$ and $r$. 
\begin{lem}
The functional $A \mapsto \cov_- (A)$ is non-decreasing for the inclusion. 
\end{lem}
\begin{proof}
We use the strong Markov property. 
The shift operator is denoted by $\sigma$:
for every integer $k$
\[
\sigma_k ( X_0 ,X_1 , \dotsc ) = (X_k , X_{k+1}, \dotsc ).
\]
Let $A\subset B$ and let $x\in B$. Then
\[
T_{cov} (B) \geq T(A) + T_{cov} (A) \circ \sigma_{T(A)} .
\]
In words: at time $T(A)$ the chain has yet to visit 
every point of $A\backslash \{X_{T(A)}\}$. 
By the strong Markov property
\[
\begin{split}
\E_x  T_{cov} (B) 
& \geq \E_x T (A) + \E_x \bigl[  \E_{X_T(A)}  T_{cov} (A)  \bigr] . \\
& \geq \E_x T(A)  + \cov_- (A) \\
& \geq \cov_- (A) , 
\end{split}
\]
which is the result. 
\end{proof}

%%%%%%%%%%%%%%%%%%%%%%%%%%%%%%%%%
\subsection*{Variations on Matthews' bound}
The following is due to Matthews~\cite{MAT}.
\begin{lem}
\label{matthews1}
Let $A$ be a finite subset of $M$, let $a>0$ and 
assume that $\E_{x} T (y) \geq a$ for every $x\neq y$ in $A$. 
Then
\[ 
 \cov_- ( A ) \geq  
 a \sum_{k=1}^{ \abs{ A} -1} \frac{1}{k} 
 \geq a \log ( \abs{A} ) . 
 \]
\end{lem}
\begin{proof}
Let $x\in A$. Assuming that $\abs{A} \geq 2$
(otherwise the result is trivial) we have
\[
\sum_{y \in A, y\neq x} \prob_x \bigl( T_{cov} (A)  = T(y) \bigr)  = 1 .
\]
So there exists $y \in A$ such that 
\begin{equation}
\label{step-mat}
\prob_{x} \bigl( T_{cov} (A) = T(y) \bigr) \geq \frac 1 {\abs{A}-1}. 
\end{equation}
Let $A' = A \backslash \{ y \}$, 
let $S =  T_{cov} ( A' )$ and let $T = T_{cov} (A)$. 
Clearly
\[
T = S + ( T(y) \circ \sigma_S ) \mathbf{1}_{ \{S < T(y)\} } .
\]
By the strong Markov property
\[
 \E_{x}   T   = \E_x S + \E_{x} \bigl[   \bigl( \E_{X_S }  T (y) \bigr)  \mathbf{1}_{\{ S < T (y) \}}  \bigr] .
\]
On the event $\{S < T(y)\}$ the point $X_S$ is an element of $A$ different from $y$.
Therefore $\E_{X_S }  T (y) \geq a$. Together with~\eqref{step-mat} we obtain
 \[
 \E_{x}  T_{cov} (A)  \geq \E_{x} T_{cov} (A') + \frac a {\abs{A}-1} .
 \]
An obvious induction on $\abs{A}$ finishes the proof. 
 \end{proof}
The following lemma is proved the same way. 
\begin{lem}
\label{matthews2}
Let $H_1,\dots, H_m$ be non-empty subsets of $M$ satisfying
\[
\E_{x} T ( y ) \geq a , \forall (x,y)\in H_i\times H_j,\ \forall i\neq j .
\]
Then for all $x\in \cup_{i\leq m} H_i$
\[ 
 \E_{x} \max_{i\leq m}  T ( H_i) 
 \geq a \log (m) . 
 \]
\end{lem}
An additional application of the strong Markov property yields the following refinement 
of the previous lemma. 
\begin{prop}
\label{matthews3}
Let $H_1,\dotsc,H_m$ be non-empty subsets of $M$
satisfying $\E_x T_y \geq a$ for all $(x,y)\in H_i\times H_j$, for all $i\neq j$.
Then
\begin{equation}
\label{growth-eq1}
\cov_- \bigl( \bigcup_{i\leq m} H_i \bigr) \geq a \log( m ) + \min_{i\leq m} \cov_- ( H_i ) .
\end{equation}
\end{prop}
\begin{proof}
Let $x \in \cup_{i\leq m} H_i$.
Let  $S = \max_{i\leq m} T(H_i)$ and $T = T_{cov} ( \cup_{i\leq m} H_i )$. 
If $S = T(H_i)$ then at time $S$ the chain has yet to visit every point of $H_i\backslash\{X_S\}$. 
Therefore
\[
T \geq S + \sum_{i=1}^m \bigl( T_{cov} (H_i) \circ \sigma_S \bigr) \mathbf{1}_{\{ S = T(H_i)\}}
\]
Using the strong Markov property, we get
\[
\begin{split}
\E_x   T 
& \geq \E_x S + \sum_{i=1}^m \E_x \bigl[  \bigl( \E_{X_S}  T_{cov} ( H_i ) \bigr) \mathbf{1}_{\{ S = T(H_i)\}}  \bigr]  \\
& \geq  \E_x S + \min_{i\leq m} \cov_- ( H_i ). 
\end{split}
\]
Together with the previous lemma we get the result. 
\end{proof}
We are close to desired growth condition. 
We would like to obtain the inequality~\eqref{growth-eq1} under the weaker hypothesis
\[
d( x,y ) = \E_x T(y) + \E_y T(x)  \geq a , \  \forall x,y \in H_i \times H_j , i\neq j .
\]
This is done in the next section.
Roughly speaking, reversibility insures
that for a reasonable proportion of $x$ and $y$ the 
hitting times $\E_x T(y)$ and $\E_y T(x)$ are of the same 
of order of magnitude.
\section*{Reversibility}
Again this part of the argument is taken from 
Kahn, Kim, Lovasz and Vu's article~\cite{KKLV}.
\\
We start with a simple lemma concerning directed graphs. 
Given a directed graph $G = (V , E)$, a path of $G$ 
is a sequence $x_1,\dotsc,x_m$ of vertices
satisfying  $(x_i , x_{i+1}) \in E$ for $i\leq m$. The length 
of such a path is defined to be $m$. An independent  
set is a subset $A$ of $V$ satisfying $(x,y)\notin E$ for all $x,y$ in $A$. 
\begin{lem}
If every path of $G$ has length at most $m$ then $G$ 
has an independent set of cardinality at least $\abs{V} / m$. 
\end{lem}
This is standard, but we still sketch the argument.
It is easy to show by induction on $m$ that 
$G$ is then $m$-colorable: it is possible
to map the vertices of $G$ to $\{1,\dots,m\}$
in such a way that connected points have different
images. 
Then by the pigeon hole
principle, at least $\abs{V} / m$ vertices have the same image,
which is the result.\\
From now on the chain $(X_n)_{n\geq 0}$ is 
assumed to be reversible. Consequently, we
have the following commuting property 
for hitting times. 
\begin{lem}
For every sequence $x_1 ,\dotsc , x_m$ of elements of $M$ we have
\begin{equation}
\label{cycle}
\begin{split}
\E_{x_1} T( x_2 ) & + \dotsb + \E_{x_{m-1}} T ( x_m) + \E_{x_m} T(x_1)  \\
& = \E_{x_1} T( x_m ) +  \E_{x_{m}} T ( x_{m-1}) + \dotsb + \E_{x_2} T(x_1)  .
\end{split}
\end{equation}
\end{lem}
We refer to~\cite[Lemma 10.10]{LPW} for a proof. 
\begin{cor}
\label{rev-cor}
Let $A$ be a subset of $M$ and $a>0$. If $\Delta ( A, d ) \leq 16 a$
and if $d(x,y) \geq a$ for all $x\neq y$ in $A$ then there exists a subset 
$A'$ of $A$ satisfying
\begin{itemize}
\item $\abs{A'} \geq \abs{A} / 33 $. 
\item $\E_x T(y) \geq a/4$ for all $x\neq y$ in $A'$. 
\end{itemize}
\end{cor}
\begin{proof}
We define a graph $G$ with vertex set $A$ by 
saying that the edge $(x,y)$ is present
if $x\neq y$ and $\E_x T (y) \leq a / 4$.
Let $x_1,\dotsc,x_m$ be a path of $G$. 
Then the inequalities
\[ 
\begin{split}
\E_{x_i} T(x_{i+1} ) & \leq a/4 \\
  \E_{x_{i+1}} T(x_i) & \geq 3a/4
\end{split}
 \]
 and equation~\eqref{cycle} give
\[
\frac {(m-1) a} 4  + \E_{x_m} T(x_1) \geq \frac{3(m-1) a}{4} + \E_{x_1} T(x_m) .
\]
Together with the bound on the diameter of $A$ we obtain $m-1 \leq 32$. 
Therefore $G$ has an independent set of cardinality at least $ \abs{A} / 33$. 
This is our set $A'$. 
\end{proof}
%%%%%%%%%%%%%%%%%%%%%%%%%%%%%%%%%%%%%
\subsection*{The growth condition for the cover time}
\begin{prop}
The functional $A \mapsto \cov_- (A) $ satisfies 
the growth condition with parameters $r=16$ and $\tau=5$. 
\end{prop}
\begin{proof}
Let $n\in \mathbb N$, let $a>0$ and $m=N_{n+5}$. Let $H_1,\dotsc,H_m$
satisfy 
\begin{enumerate}
\item $\Delta ( \cup_{i\leq m}  H_i ) \leq 16 a$.
\item $ d ( H_i , H_j ) \geq a$, for all $i\neq j$.
\item $\Delta ( H_i ) \leq a/16$ for all $ i\leq m$.
\end{enumerate}
Let $x_1,\dotsc,x_m$ belong to  $H_1, \dotsc, H_m$ respectively. 
By the first two properties and Corollary~\ref{rev-cor}, there exists a subset 
$I$ of $\{1,\dotsc,m\}$ satisfying
\begin{itemize}
\item $\abs{I} \geq m / 33$. 
\item $\E_{x_i} T (x_j ) \geq a/4$ for every $i\neq j$ in $I$.
\end{itemize}
Let $i\neq j$ in $I$ and let $(x,y) \in H_i \times H_j$. Then 
\[
\E_{x} T(y) \geq \E_{x_i} T (x_j) - \E_{x_i} T (x) - \E_y T( x_j ) \geq \frac a 4 - \frac a {16} - \frac a {16} = \frac a 8.
\]
Proposition~\ref{matthews3} gives
\[
\cov_- \bigl( \bigcup_{i\in I} H_i \bigr) \geq \frac a 8 \log ( \abs{ I} ) + \min_{i \in I} \cov_- ( H_i ). 
\]
Since
\[
 \abs{I}  \geq N_{n+5} / 33 \geq N_{n+5} / N_{3} \geq N_{n+4} \geq \e^{ 8 \cdot 2^n }  
\]
we obtain
\[
\cov_- \bigl( \bigcup_{i\leq m} H_i \bigr) \geq a 2^n  + \min_{i \leq m} \cov_- ( H_i ) , 
\]
which is the result.
\end{proof}
Then, by Theorem~\ref{growth-thm} we obtain 
\[
\gamma_1 ( M , d ) \leq L ( \cov_- ( M ) + \Delta ( M , d ) ) . 
\]
Obviously we can replace $M$ by any subset $A$ of 
$M$ in this inequality: if a functional $F$ satisfy the growth condition
on $(M,d)$ then it also satisfies it on $(A,d)$. 
\section*{Appendix}
We have seen in the introduction that for any metric space $(M,d)$ 
\begin{equation}
\label{appendix-eq}
\bigl[ \gamma_2 ( M , \sqrt{d} ) \bigl]^2 \leq  C  \log ( \log \abs{M} )  \gamma_1 (M,d) . 
\end{equation}
We show in this appendix that this is sharp and that the 
example saturating the inequality can be chosen to be 
the state space of a reversible Markov chain equipped with the 
commute distance. The example is taken from~\cite{KKLV} and was pointed
out to the author by James Lee. 
 \\
Let $M$ be a rooted tree of depth $D$ (large enough) satisfying
\begin{itemize} 
\item nodes at depth $i\leq D-1$ have $N_i +1$ children,
\item edges between depth $i$ and depth $i+1$ have multiplicity $2^i$,
\end{itemize}
and let $X$ be the random walk on this graph. 
The probability measure defined by $\pi (x) = d(x) / 2 E$ for every $x$, where 
$d(x)$ is the number of edges (counted with multiplicity) starting from $x$ and
$E$ is the total number of edges, is reversible. 
Let us compute the commute distance $d$. 
Because of the tree structure it is easily seen that
\begin{equation}
\label{tree-distance}
d ( x, y ) = \sum_{i=0}^{n-1}  d(x_i,x_{i+1})
\end{equation}
where $x_0,\dotsc,x_n$ is the shortest path from $x$ to $y$. 
Therefore it is enough to compute $d(x,y)$ when $x$ and $y$
are neighbors, in which case we use the formula (see~\cite{AF})
\[
\prob_x ( T(y) < T^1 (x) ) = \frac 1 {\pi (x) d(x,y)} . 
\] 
Because of the tree structure again $\prob_x ( T(y) < T^1 (x) )$ is just the transition 
probability from $x$ to $y$. We obtain 
\[
d( x, y ) = 2E\cdot 2^{-i}
\]
when $(x,y)$ is an edge between depth $i$ and depth $i+1$. 
When $x$ and $y$ are any two nodes of $M$, equality~\eqref{tree-distance} 
then implies that
\begin{equation}
\label{distance-depth}
E\cdot  2^{-i+1} \leq d ( x,y ) \leq  E \cdot 2^{-i+3}
\end{equation}
where $i$ is the depth of their closest common ancestor.  
\begin{prop}
There is a universal constant $C$ such that
\begin{align}
\label{example-gamma1}
\frac {D \cdot E} C & \leq \gamma_1 ( M , d )  \leq  C \cdot D \cdot E\\
\label{example-gamma2}
\frac {D \cdot \sqrt{E} } C & \leq  \gamma_2 ( M , \sqrt d ) \leq C \cdot D \cdot \sqrt{E} .
\end{align}
\end{prop}
Since $D$ is of the order of $\log ( \log \abs{M} )$, this shows that~\eqref{appendix-eq} is sharp 
(up to the constant). 
\begin{proof}
Let us start with the upper bound of~\eqref{example-gamma1}. 
It is more convenient to use the following definition for $\gamma_1$:
\[
\gamma_1 ( M , d ) = \inf \sup_{x\in M} \sum_{i=0}^{+\infty} 2^i d( x , M_i) 
\] 
where the infimum is taken over every sequence $( M_i )_{i\in \N}$ of subsets
of $M$ satisfying the cardinality condition $\abs{M_i} \leq N_i$ for every $i$. 
It is well known (see~\cite{TALA}) that this definition coincides with the one with partitions, 
up to a universal factor. 
\\
For $0\leq i \leq D$ let $S_i$ be the set of vertices of depth at most $i$. 
Using~\eqref{distance-depth} we obtain $d ( x , S_i) \leq E \cdot 2^{-i+3}$ 
for every $x\in M$. Therefore 
\[
\sup_{x\in M} \sum_{i=0}^{+\infty} 2^{i} d(x,S_i)  \leq E \sum_{i=0}^D 2^i 2^{-i+3}  = 8E\cdot  (D+1).
\]
Besides, it is easily shown that 
\[
\abs{S_i} \leq N_{i+3}  .
\]
The sequence $(S_i)_{n\in \N}$ does not quite satisfies the right cardinality condition,
but this is not a big deal. If we shift the sequence by letting $M_0=M_1=M_2=S_0$ 
and $M_i = S_{i-3}$ for $i \geq 3$, we still have
\[
\sup_{x\in M} \sum_{i=0}^{+\infty} 2^i d(x,M_i) \leq  C \cdot E \cdot D 
\]
for some universal $C$, which proves the upper bound of~\eqref{example-gamma1}.
\\
To prove the lower bound we need to show that the previous 
sequence of approximations is essentially optimal. Let $(M_i)_{i\geq 0}$
be a sequence of subsets of $M$ satisfying $\abs{M_i} \leq N_i$ for every $i$. 
A vertex $x$ of depth $i \leq D-1$ has $N_i +1$ children. So at least one them, 
call it $y$, has the following property: neither $y$ nor any of its offsprings 
belong to $M_i$. Using this observation, we can construct inductively a sequence
$x_0,x_1,\dotsc,x_D$, where $x_0$ is the root of $M$ and such that
\begin{itemize}
\item $x_{i+1}$ is a child of $x_i$,
\item neither $x_{i+1}$ nor any of its offsprings belong to $M_i$,
\end{itemize}
for every $ i \leq D-1$.
Let $i \leq D-1$ and let $x \in M_{i}$. 
Since $x$ is not an offspring of $x_{i+1}$ 
we have $d(x , x_D ) \geq  E \cdot 2^{-i+1}$. Thus
\[
\sum_{i=0}^\infty  2^i d(x_D , M_i ) \geq  E \sum_{i=0}^{D-1} 2^i 2^{-i+1} = 2 E \cdot D , 
\]
which proves the lower bound of~\eqref{example-gamma1}. 
\\
Inequality~\eqref{example-gamma2} is proved exactly the same way. 
\end{proof}
\subsection*{Acknowledgement}
This work owes a lot to stimulating discussions the author had 
with Guillaume Aubrun. 


\begin{thebibliography}{plain}

\bibitem{AF}
D.~Aldous and J.A.~Fill,
\textit{Reversible Markov chains and random walks on graphs},
in preparation, online version at the following web page.\\
\verb"http://www.stat.berkeley.edu/users/aldous/RWG/book.html"

\bibitem{BDNP}
M.T.~Barlow, J.~Ding, A.~Nachmias and Y.~Peres,
The evolution of the cover time,
Combin. Probab. Comput. 20 (3) (2011) 331--345.

\bibitem{DLP}
J.~Ding, J.R.~Lee and Y.~Peres,
Cover times, blanket times, and majorizing measures,
Ann. of Math. 175 (3) (2012) 1409--1471. 

\bibitem{KKLV}
J.~Kahn, J.H.~Kim, L.~Lov{\'a}sz and V.H.~Vu,
The cover time, the blanket time, and the {M}atthews bound,
in 
\textit{41st {A}nnual {S}ymposium on {F}oundations of {C}omputer {S}cience}, 
IEEE Comput. Soc. Press, 
Los Alamitos, 
2000, 
pp.~467--475. 

\bibitem{LPW}
D.A.~Levin, Y.~Peres and E.L.~Wilmer,
 \textit{Markov chains and mixing times},
American Mathematical Society,
Providence,
2009.
      
\bibitem{MAT}
P.~Matthews,
Covering problems for Brownian motion on spheres,
Ann. Probab. 16 (1) (1988) 189--199. 
      
\bibitem{TALA}
M.~Talagrand,
\textit{The generic chaining},
Springer Monographs in Mathematics,
Springer-Verlag,
Berlin,
2005.
\end{thebibliography}
 \end{document}